\documentclass[11pt]{article}
\usepackage{amsfonts,amssymb,amsmath}

\usepackage{color} % load the color package
   \definecolor{cites}{rgb}{0.50 , 0.00 , 0.00}  % colour for citations
   \definecolor{urls} {rgb}{0.00 , 0.00 , 0.50}  % colour for URL's
   \definecolor{links}{rgb}{0.00 , 0.00 , 0.50}   % colour for links
\usepackage[
      colorlinks=true,   % Yes, we want coloured links!
      citecolor=cites,   % Set the citations colour,
      urlcolor=urls,     % and the URL's colour,
      linkcolor=links,   % and the colour for the links.
      pdftitle={On the integer points in a lattice polytope: n-fold Minkowski sum and boundary},  % Put your title here
      pdfauthor={Marko Lindner, Steffen Roch}, % And your name here
      pdfpagemode=UseOutlines, % None, UseThumbs, UseOutlines, FullScreen
      pdfstartview=FitH,       % Fit, FitH, FitB
      bookmarksopen=false      % Bookmark tree already opened?
   ]{hyperref}

\newcommand{\cS}{{\mathcal S}}

\newcommand{\sN}{{\mathbb N}}

\newcommand{\sQ}{{\mathbb Q}}
\newcommand{\sR}{{\mathbb R}}

\newcommand{\sZ}{{\mathbb Z}}

\newcommand\eps\varepsilon

\newcommand\conv{{\rm conv}}
\newcommand\vol{{\rm vol}}
\newcommand\inte{{\rm int}}
\newcommand\R{{\mathbb R}}
\newcommand\Q{{\mathbb Q}}
\newcommand\Z{{\mathbb Z}}
\newcommand\N{{\mathbb N}}

\newcommand\T{{\mathcal T}}

\newtheorem{theorem}{Theorem}[section]
\newtheorem{lemma}[theorem]{Lemma}
\newtheorem{corollary}[theorem]{Corollary}
\newtheorem{proposition}[theorem]{Proposition}

\newenvironment{remark}
 {\par\noindent\refstepcounter{theorem}{\bf Remark \thetheorem}\ }
 {\raisebox{1mm}{\framebox{}}\pagebreak[2]}

\newenvironment{example}
 {\par\noindent\refstepcounter{theorem}{\bf Example \thetheorem}\ }
 {\raisebox{1mm}{\framebox{}}\pagebreak[2]}

\newenvironment{proof}
 {\par\noindent{\bf Proof.}}
 {\rule{2mm}{2mm}\pagebreak[2]}

\begin{document}
\title{\bf On the integer points in a lattice polytope:\\$n$-fold Minkowski sum and boundary}
\author{{\sc Marko Lindner}\quad and\quad {\sc Steffen Roch}}
\date{\today}
\maketitle
\begin{quote}
\renewcommand{\baselinestretch}{1.0}
\footnotesize {\sc Abstract.} In this article we compare the set of
integer points in the homothetic copy $n\Pi$ of a lattice polytope
$\Pi\subseteq\R^d$ with the set of all sums $x_1+\cdots+x_n$ with
$x_1,...,x_n\in \Pi\cap\Z^d$ and $n\in\N$. We give conditions on the
polytope $\Pi$ under which these two sets coincide and we discuss
two notions of boundary for subsets of $\Z^d$ or, more generally,
subsets of a finitely generated discrete group.
\end{quote}

\noindent
{\it Mathematics subject classification (2000):} 52B20; 52C07, 65J10\\
{\it Keywords and phrases:} lattice polytopes, integer points, boundary,
projection methods
\section{Introduction}
Throughout, we denote by $\sN, \, \sZ, \, \sQ$ and $\sR$ the
natural, integer, rational and real numbers, respectively, and we
fix a number $d \in \sN$. For arbitrary $n \in \sN$, we compare the
set of all integer points in the homothetic copy $n \Pi$ of a
lattice polytope $\Pi \subseteq \sR^d$ (that means $\Pi$ is the
convex hull of a finite number of points in $\sZ^d$) with the set of
sums $x_1 + \cdots + x_n$ with $x_1, \, \ldots, \, x_n \in \Pi \cap
\sZ^d$. It is easy to see that the latter set is always contained in
the first -- but in general, they are different. We give conditions
on the polytope under which these two sets coincide and we discuss
two notions of boundary for subsets of $\sZ^d$ and, more generally,
of a finitely generated (not necessarily commutative) discrete
group.

The motivation for this paper stems from the study of projection
methods for the approximate solution of operator equations. Let $A$
be a bounded linear operator acting on a Banach space. To solve the
operator equation $Au = f$ numerically, one chooses a sequence
$(Q_n)$ of projections (usually assumed to be of finite rank and to
converge strongly to the identity operator) and replaces the
equation $Au=f$ by the sequence of the linear systems $Q_n A Q_n u_n
= Q_n f$. What one expects is that, under suitable conditions, the
solutions $u_n$ of these systems converge to the solution $u$ of the
original equation. If the Banach space on which $A$ lives consists
of functions on a countable set $Y$ (think of a sequence space
$l^2(Y)$, for example) then it is convenient to choose an increasing
sequence $(Y_n)$ of finite subsets of $Y$ and to specify $Q_n$ as
the operator $P_{Y_n}$ which restricts a function on $Y$ to $Y_n$.

In this paper, we will be concerned with the case when $Y$ is a
finitely generated discrete group $\Gamma$. In case $\Gamma$ is the
additive group $\Z^d$, a typical (rather geometric) approach
\cite{RaRoSi:Book,Roch:FSM,Li:Habil,Li:FSMsubs} to design a projection method
is to fix a compact set (for example a lattice polytope) $\Pi
\subseteq \R^d$ and to consider the operator $P_{\Pi_n}$ of
restriction to (likewise, the operator of multiplication by the
characteristic function of) the set
\begin{equation} \label{eq:Pn}
\Pi_n := (n \Pi) \cap \Z^d
\end{equation}
for $n \in \sN$. If we assume that the origin is in the interior of
$\Pi$ then it follows that
\begin{equation} \label{eq:incmethod}
\Pi_n \subseteq \Pi_{n+1} \quad \textrm{for all} \; n \in \N \qquad
\textrm{and} \qquad \bigcup_{n \in \N} \Pi_n = \sZ^d,
\end{equation}
whence the sequence $(\Pi_n)_{n\in\N}$ gives rise to an increasing
sequence of finite-dimensional projection operators on $l^2(\sZ^d)$
that strongly converges to the identity operator as $n \to \infty$.

In the case of a general finitely generated group $\Gamma$, this
geometric approach is clearly infeasible. A natural idea here is to
fix a finite set $\Omega \subseteq \Gamma$ of generators of $\Gamma$
(that is, we assume that $\Omega$ generates $\Gamma$ as a
semi-group) and to consider the set
\begin{equation} \label{eq:Pn'}
\Omega_n :=  \{x_1 x_2 \cdots x_n : x_1, \, x_2, \, \dots, \, x_n \in \Omega \}
\end{equation}
of all words of length $n \ge 1$ over the alphabet $\Omega$ in place
of \eqref{eq:Pn}. If $\Omega$ is symmetric (in the sense that
$x^{-1} \in \Omega$ if $x \in \Omega$), contains the identity
element $e$ of $\Gamma$ (in analogy to the above approach in $\Z^d$)
and if $\Gamma$ is equipped with the word metric over $\Omega$ then
$\Omega_n$ is the disk of radius $n$ in $\Gamma$ around the identity
$e$. Moreover, one gets that also $(\Omega_n)_{n \in \N}$ yields an
increasing sequence of finite-dimensional projections with strong
limit identity, i.e., \eqref{eq:incmethod} holds with $\Pi_m$
replaced by $\Omega_m$ and $\Z^d$ by $\Gamma$.

A natural question before proposing the latter approach for general
finitely generated groups $\Gamma$ is whether or not the geometric
approach \eqref{eq:Pn} and the algebraic approach \eqref{eq:Pn'}
coincide if we have $\Gamma = \Z^d$ and use $\Omega := \Pi \cap
\Z^d$ as finite set of generators in \eqref{eq:Pn'} with a symmetric
lattice polytope $\Pi \subseteq \R^d$ containing the origin in its
interior. This question is discussed in Section \ref{sec:Pn-vs-Pn'}.
We will give conditions on the polytope $\Pi$ under which
\eqref{eq:Pn} and \eqref{eq:Pn'} coincide -- but in general they do
not.

In Section \ref{sec:bdry} we address a further question that arises
in the study of projection methods. In \cite{Roch:FSMgroup} it has
been pointed out that the ``boundaries'' $\partial_\Omega \Omega_n$
(if appropriately defined) of the sets $\Omega_n$ (or $\Pi_n$) hold
the key to the answer to whether or not the projection method
\[ %\begin{equation} \label{eq:FSM}
P_{\Omega_n} A P_{\Omega_n} u_n = P_{\Omega_n} f, \qquad n \in \sN,
\] %\end{equation}
yields stable approximations $u_n$ to the solution $u$ of $Au=f$. In
Section \ref{sec:bdry} we give special emphasis to the question
whether the ``algebraic boundaries'' $\partial_\Omega \Omega_n$
coincide with the corresponding ``numerical boundaries'' $\Omega_n
\setminus \Omega_{n-1}$.
\section{Lattice polytopes, enlargements and integer points} \label{sec:Pn-vs-Pn'}
Now fix $d\in\N$, let $e_1,...,e_d$ be the standard unit vectors of
$\R^d$, and denote the unit simplex $\conv \{0,e_1,...,e_d\}$ by
$\sigma_d$. (For standard notions on convex polytopes we recommend
\cite{Henk97,Gruenbaum,Ziegler}; for lattice polytopes see
\cite{BeckRobins,GritzWills,Gruber,Schrijver,White}.)

Given a non-empty subset $S$ of $\R^d$ and a positive integer $n$,
we write
\[
%\begin{equation} \label{eq:enlarge}
nS := \{ns:s\in S\}\quad\textrm{and}\quad n*S := \{s_1+\dots
+s_n:s_1,...,s_n\in S\} = S + \dots + S
%\end{equation}
\]
for the ratio-$n$ homothetic copy and $n$-fold Minkowski sum of $S$,
respectively. For convenience, we also set $0 S := \{0\}$ and $0 * S
:= \{0\}$. It is easy to see that $nS=n*S$ holds for all $n \in \N$
if $S$ is convex. Indeed, the inclusion $nS \subseteq n*S$ is always
true, and, by convexity of $S$, $s_1+\dots +s_n$ can be written as
$ns$ with $s=(s_1+\dots +s_n)/n\in S$ for all $n \in \N$ and
$s_1,...,s_n \in S$. (Note that equality of $nS$ and $n*S$ for all
$n\in\N$ does not imply convexity of $S$, as $S=\Q$ shows.)

For a convex set $\Pi\subseteq\R^d$ containing at least two integer
points, it is clear that $(n\Pi)\cap\Z^d\ne n(\Pi\cap\Z^d)$ as soon
as $n>1$. But (as motivated in the introduction) a much more
interesting question is whether or not
\begin{equation} \label{eq:2enlargements}
(n\Pi)\cap\Z^d\ =\ n*(\Pi\cap\Z^d)
\end{equation}
is true for all $n\in\N$. We will study this question for certain
polytopes $\Pi$.

Let $v_1,...,v_k$ be points of $\Z^d$ with their affine hull equal
to $\R^d$ and put $\Pi=\conv\{v_1,...,v_k\}$. $\Pi$ is a so-called
{\sl lattice polytope} as all its vertices are in $\Z^d$. We will
suppose that there is no proper subset $I$ of $\{1,...,k\}$ with
$\Pi=\conv\{v_i:i\in I\}$, so that $v_1,...,v_k$ are the vertices of
$\Pi$. If $\Pi\cap\Z^d$ only consists of the vertices of $\Pi$ then
$\Pi$ is called an {\sl elementary polytope}
\cite{Kantor98,Schrijver} (or a {\sl lattice-(point-)free polytope}
\cite{Henk98,Kantor99}). The following lemma is fairly standard:

\begin{lemma} \label{lem:fullyoccupied}
Let $A\in\Z^{d\times d}$ be a matrix and $a_1,...,a_d\in\Z^d$ its columns.

{\bf a)} The following conditions are equivalent:

\begin{tabular}{rl}
  (i)&$A(\Z^d)=\Z^d$.\\
 (ii)&$A^{-1}$ is an integer matrix.\\
(iii)&$\det A=\pm 1$.\\
 (iv)&The parallelotope $A([0,1]^d)$ spanned by $a_1,...,a_d$ has volume $1$.\\
  (v)&The parallelotope $A([0,1]^d)$ is elementary.
\end{tabular}

{\bf b)} The condition

\begin{tabular}{rl}
 (vi)&The simplex $A(\sigma_d)=\conv\{0,a_1,...,a_d\}$ is elementary.
\end{tabular}

is necessary for $(i)$--$(v)$; it is moreover sufficient iff $d\in\{1,2\}$.
\end{lemma}

If $(i)$--$(v)$ hold then $A$ is called an {\sl integer unimodular
matrix} \cite{Gruber,Schrijver} and the simplex $A(\sigma_d)$ has
volume $1/d!$ and is sometimes called a {\sl primitive} (or {\sl unimodular})
{\sl simplex} (e.g. \cite{Kantor98}). So primitive simplices are elementary,
and the converse holds iff $d\in\{1,2\}$.

For the sake of completeness, we give a short sketch of the proof
of Lemma \ref{lem:fullyoccupied}:

\begin{proof}
Part a) follows by standard arguments using $\det A^{-1} = 1/\det A$
and $\vol(A([0,1]^d))=|\det A|\,\vol([0,1]^d)$. The implication
$(v)\Rightarrow (vi)$ holds by $\sigma_d\subseteq [0,1]^d$. For
$d=1$, the implication $(vi)\Rightarrow (v)$ is clear by $\sigma_1 =
[0,1]^1$. For $d=2$, if $x$ is an integer non-vertex point in
$A([0,1]^2)$, then also $a_1+a_2-x$ is an integer non-vertex point
in $A([0,1]^2)$. But one of the two points is in $A(\sigma_2)$, so
that $(vi)\Rightarrow (v)$ holds. For $d\ge 3$, there are elementary
but not primitive simplices (see Examples \ref{ex:simpl} a and b
below).
\end{proof}

Here are two slightly different constructions leading to elementary
but not primitive simplices in dimension $d\ge 3$.

\begin{example} \label{ex:simpl}
{\bf a) } Let $d\ge 3$, fix an $m\in\N$, take $a_1:=e_1,\ a_2:=e_2,\
...\ ,\,a_{d-1}:=e_{d-1}\in\Z^d$ and $a_d:=(-1,...,-1,m)^\top$, and
let $A\in\Z^{d\times d}$ be the matrix with columns $a_1,...,a_d$.
Then
\[
\Sigma_{d,m}\ :=\ A(\sigma_d)\ =\ \conv\{0,a_1,...,a_d\}
\]
is primitive iff $m=\det A=1$. On the other hand, the number of
integer points in $\Sigma_{d,m}$ apart from its $d+1$ vertices is
equal to $k=\lfloor m/d\rfloor$ (integer division), and these $k$
other integer points are $1e_d,...,ke_d$. (To see this, look at the
projection of $\Sigma_{d,m}$ to the hyperplane spanned by
$e_1,...,e_{d-1}$.) So for $m\in\{2,...,d-1\}$ we have an elementary
but not primitive simplex.

{\bf b) } If we change the last column in the above example from
$a_d=(-1,...,-1,m)^\top$ to $a'_d:=(1,...,1,m)^\top$ with $m\in\N$
and call the new matrix $A'$ then, again, the simplex
\[
\Sigma_{d,m}'\ :=\ A'(\sigma_d)\ =\ \conv\{0,a_1,...,a_{d-1},a'_d\}
\]
is not primitive for $m=\det A'\ge 2$ but now it is elementary for
all $m\in\N$. So this simplex $\Sigma'_{d,m}$ can have arbitrarily
large volume $m/d!$ without containing any integer points other than
its vertices. The simplices $\Sigma'_{d,m}$ were first considered
(because of this property) by Reeve \cite{Reeve} in case $d=3$
and have since been termed {\sl Reeve simplices}. There are results
that relate the maximal volume of an elementary polytope in $\R^d$
to its surface area \cite{BHW} or its inradius \cite{Henk98}.
\end{example}

Given a full-dimensional lattice polytope $\Pi\subseteq\R^d$ and a
set $\T$ of full-dimensional lattice simplices $S_1,...,S_m\subseteq
\Pi$ with
\[
\bigcup_{i=1}^m S_i\ =\ \Pi\qquad\textrm{and}\qquad
S_i\cap S_j\ \textrm{is a face of both $S_i$ and $S_j$},\ \forall i,j,
\]
then the set $\T=\{S_1,...,S_m\}$ is called a {\sl triangulation} of
$\Pi$. The triangulation $\T$ is called {\sl elementary} or {\sl
primitive} if all its elements $S_i$ are, respectively, elementary
or primitive simplices.

Here is our main result on the equality \eqref{eq:2enlargements}:

\begin{proposition} \label{prop:main}
If a full-dimensional lattice polytope $\Pi\subseteq\R^d$ possesses
a primitive triangulation then equality \eqref{eq:2enlargements}
holds for all $n\in\N$.
\end{proposition}
\begin{proof}
Let $n\in\N$ and $x\in n*(\Pi\cap\Z^d)$. Then $x=p_1+...+p_n$ for
some $p_1,...,p_n\in \Pi\cap\Z^d$, so that $x\in\Z^d$ and $x=np$
with $p=(p_1+...+p_n)/n$. But $p\in \Pi$ by convexity of $\Pi$.

Now let $x\in (n\Pi)\cap\Z^d$, so $x=np$ is an integer point with
$p\in \Pi$. Let $\T=\{S_1,...,S_m\}$ be a primitive triangulation of
$\Pi$ and let $w_0,...,w_d$ be the vertices of a simplex $S_i$ that
contains $p$. Now there is a unique way to write $p$ as a convex
combination of $w_0,...,w_d$. So there are $\alpha_0,...,\alpha_d\in
[0,1]$ so that $p=\alpha_0 w_0+...+\alpha_d w_d$ and
$\alpha_0+...+\alpha_d=1$. Together with $x=np$ this implies
\begin{equation} \label{eq:convcomb}
\left(\begin{array}{cccc} |&|& & |\\
w_0&w_1&\cdots&w_d\\
| & |& & |\\[0.8ex] \hline \\[-2ex]1 & 1 &\cdots & 1
\end{array}\right)
\left(\begin{array}{c} n\alpha_0\\ n\alpha_1\\ \vdots\\
n\alpha_d\end{array}\right) \ =\ n\left(\begin{array}{c} |\\p\\
|\\[0.8ex] \hline \\[-2ex] 1\end{array}\right)
\ =\ \left(\begin{array}{c} |\\x\\
|\\[0.8ex] \hline \\[-2ex] n\end{array}\right).
\end{equation}
If we refer to the matrix in \eqref{eq:convcomb} as $M$ then, after
subtracting the first column from all the others and then expanding
by the last row,
\begin{eqnarray*}
\det M &=& \det \left(\begin{array}{cccc} |&|& & |\\
w_0&w_1-w_0&\cdots&w_d-w_0\\
| & |& & |\\[0.8ex] \hline \\[-2ex]1 & 0 &\cdots & 0
\end{array}\right)\\
& = & (-1)^{d+1} \det \left(\begin{array}{ccc} |& & |\\
w_1-w_0&\cdots&w_d-w_0\\
|& & |
\end{array}\right)\ \in\ \{\pm 1\}
\end{eqnarray*}
by Lemma \ref{lem:fullyoccupied} since $w_1-w_0,\,\cdots,w_d-w_0$
span the primitive simplex $S_i-w_0$. So $M^{-1}$ exists and is an
integer matrix. By \eqref{eq:convcomb}, it follows that
$\beta_0:=n\alpha_0,\,\cdots\,,\,\beta_d:=n\alpha_d$ are integers
since $M^{-1}$ and $x$ have integer entries. Summarizing, we get
that
\begin{equation} \label{eq:herewego}
\beta_0 w_0\,+\,\beta_1 w_1\,...\,+\,\beta_d w_d\ =\ x,
\end{equation}
where $\beta_0,...,\beta_d\in\{0,...,n\}$ and
$\beta_0+...+\beta_d=n$, so that \eqref{eq:herewego} is the desired
decomposition of $x$ into a sum of $n$ elements from $\Pi\cap\Z^d$.
\end{proof}

It is not clear to us whether the existence of a primitive triangulation
is necessary for equality \eqref{eq:2enlargements} to hold for all $n\in\N$.
(Is it possible that every $p\in\Pi$ is contained in a primitive simplex
$S(p)\subset\Pi$ without the existence of a ``global'' primitive triangulation
of $\Pi$?)

%\noindent\rule{\textwidth}{1pt} {\tt Ich bin nicht sicher, ob die
%primitive Triangularbarkeit auch notwendig f\"ur
%\eqref{eq:2enlargements} ist. K\"onnte es z.B. passieren, dass jeder
%Punkt $p\in \Pi$ in einem primitiven Simplex $S\subseteq \Pi$ liegt
%ohne dass es eine ``globale'' primitive Triangulierung von $\Pi$ gibt?}\\
%\rule{\textwidth}{1pt}

It is not hard to see that every lattice polytope $\Pi$ possesses an
elementary triangulation. (Existence of a triangulation can be shown
by induction over the number of vertices of $\Pi$, and every
non-elementary simplex $S_i$ can be further triangulated with
respect to its integer non-vertex points.) Existence of a primitive
triangulation however is a different question -- at least in
dimensions $d\ge 3$.

\begin{corollary}
If $\Pi$ is a full-dimensional lattice polytope in $\R^d$ with
$d\in\{1,2\}$ then equality \eqref{eq:2enlargements} holds for all
$n\in\N$.
\end{corollary}
\begin{proof}
Every lattice polytope has an elementary triangulation. In
dimensions $d\in\{1,2\}$, by Lemma \ref{lem:fullyoccupied} b), an
elementary triangulation is always primitive. Now apply Proposition
\ref{prop:main}.
\end{proof}

In dimension $d\ge 3$, it is generally a difficult question whether
or not a given lattice polytope $\Pi$ has a primitive triangulation
(see e.g. \cite{BetkeKneser,Kantor98}). The simplices $\Sigma_{d,m}$
in Example \ref{ex:simpl} a) with $m\in\{2,...,d-1\}$ and
$\Sigma_{d,m}'$ in \ref{ex:simpl} b) with $m\in\{2,3,...\}$ are
examples of lattice polytopes that have no primitive triangulation.
They are also examples, where \eqref{eq:2enlargements} is not valid
for general $n\in\N$. For example, $2\Sigma_{3,2}$ contains $e_3$,
which cannot be written as the sum of two integer points from
$\Sigma_{3,2}$. It is even possible to give examples $\Pi$ where,
for a given $k\in\N$, \eqref{eq:2enlargements} starts to fail at
$n>k$ while holding true for $n=1,2,...,k$. As an example, take
$\Pi=\Sigma_{2k+1\,,\,2}$.

The more specific question posed in the introduction is whether the
fact that $\Omega:=\Pi\cap \Z^d$ ~ (i) contains the origin, (ii) is
symmetric (i.e. $-x\in\ \Omega$ if $x\in \Omega$), and (iii)
generates $\Z^d$, i.e.
\begin{equation} \label{eq:generates}
\bigcup_{n\in\N} n*\Omega\ =\ \Z^d,
\end{equation}
guarantees equality \eqref{eq:2enlargements} for all $n\in\N$. But
also that has to be answered in the negative, as the example
$\Pi=\conv(\Sigma_{3,3}\,\cup\,-\Sigma_{3,3})=\conv\{\pm e_1\,,\,
\pm e_2\,,\, \pm(-1,-1,3)^\top\}\subseteq\R^3$ shows. Indeed,
$(-1,-1,1)^\top$ is in $2\Pi$ but not in $2*\Omega$ with
$\Omega=\Pi\cap\Z^3=\{\pm e_1\,,\, \pm e_2\,,\, \pm
e_3\,,\,\pm(-1,-1,3)^\top\,,\, 0\}$.

After all these examples, here are some results on the positive
side:

The symmetric hypercube $[-1,1]^d$ is the union of $2^d$ shifted
copies of $[0,1]^d$, each of which has a primitive triangulation.
The cross-polytope $\conv\{\pm e_1,...,\pm e_d\}$ triangulates into
$2^d$ unit simplices around the origin. Much more involved, there is
the following result by Kempf, Knudsen, Mumford and Saint-Donat
\cite{KMSD}:

\begin{lemma} \label{lem:triang}
For every lattice polytope $\Pi\subseteq\R^d$, there is an integer
$k\in\N$ such that $k\Pi$ possesses a primitive triangulation.
\end{lemma}

Consequently, for every lattice polytope $\Pi$, there is a $k\in\N$
such that $k\Pi$ satisfies \eqref{eq:2enlargements} in place of
$\Pi$ for all $n\in\N$.

\begin{remark}
The two conditions \eqref{eq:generates} with $\Omega:=\Pi\cap\Z^d$,
which says that $\Omega$ generates all of $\Z^d$, and
$0\in\inte(\Pi)$, which implies that $n*\Omega\subseteq
(n+1)*\Omega$, are connected with each other and with the validity
of \eqref{eq:2enlargements} for all $n\in\N$. Firstly, if
\eqref{eq:2enlargements} holds for all $n\in\N$ and $0\in\inte(\Pi)$
then
\[
\bigcup_{n\in\N} n*(\Pi\cap\Z^d)\ =\ \bigcup_{n\in\N} (n\Pi)\cap\Z^d
\ =\ \left(\bigcup_{n\in\N} n\Pi\right)\cap\Z^d\ =\ \Z^d
\]
 so that \eqref{eq:generates} holds.
On the other hand, from \eqref{eq:generates} it follows, by the
trivial inclusion ``$\supseteq$'' in \eqref{eq:2enlargements}, that
\[
\Z^d\ =\ \bigcup_{n\in\N} n*(\Pi\cap\Z^d)\ \subseteq\
\bigcup_{n\in\N} (n\Pi)\cap\Z^d \ =\ \left(\bigcup_{n\in\N}
n\Pi\right)\cap\Z^d\ \subseteq\ \Z^d,
\]
whence, by the convexity of $\Pi$, $\cup\, n\Pi=\R^d$ and hence
$0\in\inte(\Pi)$.
\end{remark}
\section{Boundaries of subsets of a group $\Gamma$} \label{sec:bdry}
Let $\Gamma$ be a finitely generated discrete group with identity
element $e$. We are going to introduce some notions of topological
type. Note that the standard topology on $\Gamma$ is the discrete
one; so every subset of $\Gamma$ is open with respect to this
topology.

Let $\Omega$ be a finite subset of $\Gamma$ which contains the
identity element $e$ and which generates $\Gamma$ as a semi-group,
i.e., if we set $\Omega_0 := \{e\}$ and if we let $\Omega_n$ denote
the set of all words of length at most $n$ with letters in $\Omega$
for $n \ge 1$, then $\cup_{n \ge 0} \Omega_n = \Gamma$. Note also
that the sequence $(\Omega_n)$ is increasing; so the operators
$P_{\Omega_n}$ can play the role of the finite section projections
$P_{Y_n}$ from the introduction, and in fact we will obtain some of
the subsequent results exactly for this sequence.

With respect to $\Omega$, we define the following
``algebro-topological'' notions. Let $A \subseteq \Gamma$. A point
$a \in A$ is called an {\sl $\Omega$-inner} point of $A$ if $\Omega
a := \{ \omega a : \omega \in \Omega \} \subseteq A$. The set
$\mbox{int}_\Omega A$ of all $\Omega$-inner points of $A$ is called
the {\sl $\Omega$-interior} of $A$, and the set $\partial_\Omega A
:= A \setminus \mbox{int}_\Omega A$ is the {\sl $\Omega$-boundary}
of $A$. Note that we consider the $\Omega$-boundary of a set always
as a part of that set. (In this point, the present definition of a
boundary differs from other definitions in the literature; see
\cite{Ada1} for instance.) One easily checks that
\begin{equation} \label{e030309.5}
\Omega_{n-1}\, \subseteq\, \mbox{int}_\Omega \Omega_n\, \subseteq\,
\Omega_n \quad \mbox{and} \quad
\partial_\Omega \Omega_n\, \subseteq\, \Omega_n \setminus \Omega_{n-1}
\end{equation}
for each $n \ge 1$.

Recall from \cite{Roch:FSMgroup} that there are at least two reasons
for the interest in the boundaries $\partial_\Omega \Omega_n$:
\begin{itemize}
\item
The sequence $(P_{\partial_\Omega \Omega_n})_{n \ge 1}$ belongs to
the $C^*$-algebra $\cS ({\sf Sh} (\Gamma))$ which is generated by
all finite sections sequences $(P_{\Omega_n} A P_{\Omega_n})_{n \ge
1}$ where $A$ runs through the operators on $l^2(\Gamma)$ of left
shift by elements in $\Gamma$ (i.e., they are given by the
left-regular representation of $\Gamma$ on $l^2(\Gamma)$), and it
generates the quasicommutator ideal of that algebra.
\item
There is a criterion for the stability of sequences in $\cS ({\sf
Sh} (\Gamma))$ which can be formulated by means  of limit operators,
and it turns out that it is sufficient to consider limit operators
with respect to sequences taking their values in the boundaries
$\partial_\Omega \Omega_n$.
\end{itemize}
For details, see \cite{Roch:FSMgroup}. In many instances one
observes that the ``algebraic'' boundary $\partial_\Omega \Omega_n$
coincides with the ``numerical'' boundary $\Omega_n \setminus
\Omega_{n-1}$; in fact, one inclusion holds in general as mentioned
in (\ref{e030309.5}). We will see now that the reverse inclusion can
be guaranteed if $\Gamma = \sZ^d$ and if $\Omega$ arises from a
lattice polytope $\Pi$ such that \eqref{eq:2enlargements} holds for
all $n\in\N$.
\begin{proposition} \label{01032010:p1}
Let $\Pi$ be a lattice polytope in $\sZ^d$ which satisfies
\eqref{eq:2enlargements} and set $\Omega := \Pi \cap \sZ^d$. Then
\begin{equation} \label{01032010:e2}
\partial_\Omega (n * \Omega) = (n * \Omega) \setminus ((n-1) * \Omega)
\end{equation}
holds for all positive integers $n$.
\end{proposition}
\begin{proof} As mentioned above, it is sufficient to show that
\[
(n * \Omega) \setminus ((n-1) * \Omega) \subseteq \partial_\Omega (n * \Omega).
\]
We start with working on the continuous level and check first the implication
\begin{equation} \label{01032010:e3}
\mbox{If} \quad x \in n \Pi \setminus (n-1) \Pi, \quad \mbox{then} \quad \Pi + x \not\subset n \Pi.
\end{equation}
Indeed, write $x$ as $t \omega$ with $\omega \in \partial \Pi$ (=
the usual topological boundary of $\Pi$) and $n-1 < t \le n$. Then
$x + \omega = (t+1) \omega$ with $t+1 > n$, whence $x + \omega
\not\in n \Pi$.

In the next step we show that $\omega$ can be chosen such that $x +
\omega$ becomes a grid point for $x$ a grid point. Indeed, let $x
\in \left( n \Pi \setminus (n-1) \Pi \right) \cap \sZ^d$. Consider
the points $x + \omega_i$ where the $\omega_i$, $i = 1, \ldots, k$,
run through the (integer) vertices of $\Pi$. If we would have $x +
\omega_i \in n \Pi$ for each $i$, then we would have
\[
x + \Pi = \conv \{x + \omega_i : i = 1, \ldots, k \} \subseteq n \Pi
\]
by convexity of $\Pi$, which contradicts (\ref{01032010:e3}). Hence,
for each $x \in \left( n \Pi \setminus (n-1) \Pi \right) \cap
\sZ^d$, there is a vertex $\omega_i$ of $\Pi$ such that
\[
x + \omega_i \in \left( (n+1) \Pi \setminus n \Pi \right) \cap \sZ^d.
\]
Employing the assumption \eqref{eq:2enlargements} we conclude that,
for each $x \in n* \Omega \setminus (n-1)* \Omega$ there is a
$\omega_i \in \Omega$ such that $x + \omega_i \in (n+1) * \Omega
\setminus n * \Omega$. Hence, $x$ is in the $\Omega$-boundary of $n
* \Omega$.
\end{proof}

We proceed with an example which shows that the generalized version
of (\ref{01032010:e2}),
\begin{equation} \label{01032010:e2a}
\partial_\Omega \Omega_n = \Omega_n \setminus \Omega_{n-1}\,,
\end{equation}
does not hold for general subsets $\Omega$ of a finitely generated
discrete group $\Gamma$ and $n\in\N$. Consider the matrices
\[
\omega_0 := \begin{pmatrix}
1 & 0 \\
0 & 1
\end{pmatrix}, \qquad
\omega_1 := \begin{pmatrix}
0 & 1 \\
1 & 0
\end{pmatrix}, \qquad
\omega_2 := \begin{pmatrix}
1 & 1 \\
0 & 1
\end{pmatrix}
\]
and
\[
\omega_3 := \omega_2 \omega_1 = \begin{pmatrix}
1 & 1 \\
1 & 0
\end{pmatrix}, \qquad
\omega_4 := \omega_2^{-1} = \begin{pmatrix}
1 & -1 \\
0 & 1
\end{pmatrix}, \qquad
\omega_5 := \omega_4 \omega_1 = \begin{pmatrix}
-1 & 1 \\
1 & 0
\end{pmatrix}.
\]
Then $\Omega := \{\omega_i : i = 0, \ldots, 5 \}$ generates the
group $GL(2, \, \sZ)$ as a semi-group (clearly, $\Omega$ is not
minimal as a generating system: $\omega_0, \, \omega_1, \, \omega_2,
\, \omega_4$ already generate this group). One easily checks that
$\omega_0 \omega_1 = \omega_1, \; \omega_1 \omega_1 = \omega_0, \;
\omega_2 \omega_1 = \omega_3, \; \omega_3 \omega_1 = \omega_2, \;
\omega_4 \omega_1 = \omega_5$ and $\omega_5 \omega_1 = \omega_4$,
whence $\Omega \omega_1 \subseteq \Omega$. Thus, $\omega_1 \in
\Omega_1 \setminus \Omega_0$, but $\omega_1 \not\in \partial_\Omega
\Omega_1$. So, (\ref{01032010:e2a}) is violated already for $n = 1$.

%\noindent\rule{\textwidth}{1pt}
%{\tt Ich kenne kein entsprechendes Beispiel im Kommutativen und weiss daher nicht, ob man in obiger Proposition auf die Annahme verzichten kann.}\\
%\rule{\textwidth}{1pt}

Let us conclude with a curious consequence of the coincidence
(\ref{01032010:e2}) of the boundaries. We mentioned above that the
sequence $(P_{\partial_\Omega \Omega_n})_{n \ge 1}$ always belongs
to the $C^*$-algebra $\cS ({\sf Sh} (\Gamma))$ generated by the
finite sections sequences $(P_{\Omega_n} A P_{\Omega_n})_{n \ge 1}$
where $A$ is constituted by operators of left shift by elements of
$\Gamma$. Under the conditions of Proposition \ref{01032010:p1}, we
conclude that the sequence $(P_{\Omega_n} - P_{\Omega_{n-1}})$
belongs to $\cS ({\sf Sh} (\sZ^d))$. In particular, the sequence
$(P_{\Omega_{n-1}}) = (P_{\Omega_n}) - (P_{\Omega_n} -
P_{\Omega_{n-1}})$ belongs to $\cS ({\sf Sh} (\sZ^d))$.
Consequently, with each sequence $(P_{\Omega_n} A P_{\Omega_n})_{n
\ge 1}$, the sequence
\[
(P_{\Omega_{n-1}}) \, (P_{\Omega_n} A P_{\Omega_n}) \, (P_{\Omega_{n-1}})
= (P_{\Omega_{n-1}} A P_{\Omega_{n-1}})
\]
(with the operators $P_{\Omega_{n-1}} A P_{\Omega_{n-1}}$ considered
as acting on the range of $P_{\Omega_n}$) also belongs to $\cS ({\sf
Sh} (\sZ^d))$. In particular, the algebra $\cS ({\sf Sh} (\sZ^d))$
contains a shifted copy (hence, infinitely many shifted copies) of
itself. The same fact clearly holds for every algebra which is
generated by finite sections sequences $(Q_n A Q_n)$ and contains
the sequence $(Q_n - Q_{n-1})$. A less trivial example where this
happens is the algebra of the finite sections method for operators
in (a concrete representation of) the Cuntz algebra ${\sf O}_N$ with
$N \ge 2$ \cite{Roch:Cuntz}.

\bigskip

\noindent {\bf Authors:}\\[4mm]
Marko Lindner\hfill {\tt marko.lindner@mathematik.tu-chemnitz.de}\\
TU Chemnitz\\
Fakult\"at Mathematik\\
D-09107 Chemnitz\\
GERMANY\\[8mm]
Steffen Roch\hfill {\tt roch@mathematik.tu-darmstadt.de}\\
TU Darmstadt\\
Fachbereich Mathematik\\
Schlossgartenstr. 7\\
D-64289 Darmstadt\\
GERMANY

\end{document}